\documentclass[11pt,a4paper]{amsart}
\let\amsmarkboth\markboth    %%%%%%%%Bug tetex 3.0 !!!!!
\usepackage{amsmath,amsfonts,amssymb,amsthm,amscd}
\usepackage[english]{babel}
\usepackage[mathscr]{eucal}
\usepackage{graphicx}
\usepackage{appendix}
\usepackage{color}

\makeatletter
\let\markboth\amsmarkboth
\bbl@redefine#1#2{%
   \def\bbl@arg{#1}%
   \ifx\bbl@arg\@empty
     \toks@{}%
   \else
     \toks@{\noexpand\foreignlanguage{%
              \languagename}{%
              \noexpand\bbl@restore@actives#1}}%
   \fi
   \def\bbl@arg{#2}%
   \ifx\bbl@arg\@empty
     \toks8{}%
   \else
     \toks8{\noexpand\foreignlanguage{%
              \languagename}{%
              \noexpand\bbl@restore@actives#2}}%
   \fi
   \edef\bbl@tempa{\the\toks@}%
   \edef\bbl@tempb{\the\toks8}%
   \protected@edef\bbl@tempa{%
     \noexpand\org@markboth{\bbl@tempa}{\bbl@tempb}}%
   \bbl@tempa
}
\DeclareRobustCommand*\ams@disablelinebreak{\def\\{ \ignorespaces}}
\def\maketitle{\par
   \@topnum\z@ %
   \@setcopyright
   \thispagestyle{firstpage}%
   \uppercasenonmath\shorttitle
   \ifx\@empty\shortauthors \let\shortauthors\shorttitle
   \else \andify\shortauthors
   \fi
   \@maketitle@hook
   \begingroup
   \@maketitle
   \toks@\@xp{\shortauthors}\@temptokena\@xp{\shorttitle}%
   \protected@edef\@tempa{%
     \@nx\markboth{\ams@disablelinebreak
       \@nx\MakeUppercase{\the\toks@}}{\the\@temptokena}}%
   \@tempa
   \endgroup
   \c@footnote\z@
   \@cleartopmattertags
}

\makeatother

\numberwithin{equation}{section}
\newcommand{\N}{{\mathbb N}}
\newcommand{\R}{{\mathbb R}}

\newcommand{\RR}{{\mathbb R}^2}

\newtheorem{theorem}{Theorem}[section]
\newtheorem{proposition}[theorem]{Proposition}
\newtheorem{lemma}[theorem]{Lemma}

\theoremstyle{definition}
\newtheorem{definition}[theorem]{Definition}

\theoremstyle{remark}

\newcommand{\dt}{\partial_t}

\newcommand{\eps}{\varepsilon}
\newcommand{\fe}{f_\eps}
\newcommand{\loc}{\text{loc}}

\begin{document}

%% title page

\title[Gyrokinetic limit]{On the gyrokinetic limit for the two-dimensional Vlasov-Poisson system}

\author[Evelyne Miot]{Evelyne Miot}
\address[E. Miot]{CNRS - Institut Fourier, Universit\'e Grenoble-Alpes}  \email{evelyne.miot@univ-grenoble-alpes.fr}

%\subjclass[2010]{Primary 35Q83; Secondary  35A02, 35A05, 35A24}
%\keywords{Vlasov-Poisson, gyrokinetic limit, Euler equation}

\date{\today}

\begin{abstract}
We investigate the gyrokinetic limit for the two-dimensional Vlasov-Poisson system
in the regime studied by Golse and Saint Raymond \cite{golse-sr} and by Saint-Raymond \cite{SR-02}. We
present another proof of the convergence towards the
Euler equation under several assumptions on the energy and on the $L^\infty$ norms of the initial data.
\end{abstract}

\maketitle

\section{Introduction and main results}
The purpose of this paper is to investigate an asymptotic regime for the following Vlasov-Poisson system as $\eps$ tends to zero:
\begin{equation}
\label{syst:VP}
\begin{cases}
\displaystyle \dt \fe+\frac{v}{\eps}\cdot \nabla_x \fe+\left(\frac{E_\eps}{\eps}+\frac{v^\perp}{\eps^2}\right)\cdot \nabla_v \fe=0,\quad (t,x,v)\in \R_+\times \R^2\times \R^2\\
\displaystyle E_\eps(t,x)=\int_{\R^2}\frac{x-y}{|x-y|^2}\rho_\eps(t,y)\,dy,\quad \displaystyle \rho_\eps(t,x)=\int_{\R^2}\fe(t,x,v)\,dv,\\
\fe(0,x,v)=f_{\eps}^0(x,v).
\end{cases}
\end{equation}   
Here, $f_\eps=f_\eps(t,x,v):\R_+\times \R^2\times \R^2\to \R_+$ stands for the density of a two-dimensional distribution of electric particles, called a plasma. The evolution of the plasma in the plane is submitted to the self-consistent electric field $E_\eps(t,x)$ and to a large external and constant magnetic field, orthogonal to the plane, which is represented by the term $v^\perp=(v_1,v_2)^\perp=(-v_2,v_1)$. The limit $\eps \to 0$ corresponds to the situation where the strength of the magnetic field tends to infinity. In the periodic setting, namely $(x,v)\in \mathbb{T}\times \RR$, the gyrokinetic limit was studied by Golse and Saint-Raymond \cite{golse-sr}, then by Saint-Raymond \cite{SR-02}, and also by Brenier \cite{brenier} in 
a different regime. In particular, Golse and Saint-Raymond proved that 
under suitable bounds on the initial data\footnote{See  \eqref{hyp:ini-0}, \eqref{hyp:ini-1} and \eqref{hyp:ini-2-sr} below.}, 
the sequence of spatial densities $(\rho_\eps)_{\eps >0}$ is relatively compact in\footnote{Here, $\mathcal{M}^+(\RR)$ denotes the space of bounded, positive Radon measures 
on $\RR$.} $L^\infty(\R_+,\mathcal{M}^+(\mathbb{T}\times \RR))$ weakly $\ast$ and that any accumulation point $\rho$  
is a measure-valued solution\footnote{In a sense that will be specified in Definition \ref{def:delort} below.}to the 2D Euler equation for the vorticity:
\begin{equation}
\label{syst:Euler}
\begin{cases}
\displaystyle \partial_t\rho+E^\perp\cdot \nabla \rho=0\\
\displaystyle E^\perp=2\pi \nabla^\perp\Delta^{-1}\rho.
\end{cases}
\end{equation}

\medskip

The main result of this paper will concern initial densities $f_\eps^0$ satisfying the following assumptions:
\begin{equation}\label{hyp:ini-0}
 f_\eps^0\in L^1(\RR)\cap L^\infty(\RR),\quad f_\eps^0\geq 0,\quad f_\eps^0\text{ is compactly supported}.
\end{equation}
Moreover, defining for $f\in L^1$ and $\rho=\int f\,dv$ the energy \begin{equation*}
\mathcal{H}(f)=\frac{1}{2}\iint_{\R^2\times \R^2}|v|^2f(x,v)\,dx\,dv-\frac{1}{2}\iint_{\R^2\times \RR}
\ln|x-y|\rho(x)\rho(y)\,dx\,dy,
\end{equation*}  we will assume that
\begin{equation}\label{hyp:ini-1}\begin{split}
&\sup_{\eps>0} \left( \|f_\eps^0\|_{L^1}+\int_{\RR}|x|^2\rho_\eps^0(x)\,dx \right)<+\infty,\\
&\sup_{\eps>0} \mathcal{H}(f_\eps^0)<+\infty.
\end{split}
\end{equation}Finally,
\begin{equation}
\label{hyp:ini-2}
 \eps^2\Theta\left(  \|f_\eps^0\|_{L^\infty}\right)\to 0,\quad \text{as }\eps\to 0,
\end{equation}
where $\Theta(\tau)=\tau\ln(\tau+2).$

Adapting the classical Cauchy theory for the Vlasov-Poisson equation \cite{Arsenev, loeper, ukai} for any $\eps>0$, one obtains a unique global weak solution 
$f_\eps$ to \eqref{syst:VP} belonging to $L^\infty(\R_+,L^1\cap L^\infty(\RR))$, compactly supported, such that $f_\eps(0)=f_\eps^0$. In particular, the associated spatial density $\rho_\eps$ belongs to 
$L^\infty_{\loc}(\R_+,L^\infty(\RR))$. Finally, the energy and the $L^p$ norms of the solution are non-increasing in time.

Our main result  on the asymptotics of \eqref{syst:VP} can be then stated as follows:
\begin{theorem}\label{thm:main}
Let $f_\eps^0$ satisfy \eqref{hyp:ini-0}, \eqref{hyp:ini-1} and \eqref{hyp:ini-2}. 
Let $f_\eps$ be the corresponding global weak solution to \eqref{syst:VP}.
There exists a subsequence $\eps_n\to 0$ as $n\to +\infty$ such that 
$(\rho_{\eps_n})_{n\in \N}$ converges to $\rho$ in\footnote{Here, $\rho \in C(\R_+,\mathcal{M}^+(\RR)-w^\ast)$ if and only if $\rho(t)\in \mathcal{M}^+(\RR)$ for all $t\in \R_+$ and moreover, $t\mapsto 
\int \phi(x)\,d\rho(t,x)$ is continous, for all $\varphi\in C_c(\RR)$.} $C(\R_+,\mathcal{M}^+(\R^2)-w^\ast)$. Moreover, $\rho$ belongs to $L^\infty(\R_+,H^{-1}(\RR))$ and it is a global 
solution of the 2D Euler equation\eqref{syst:Euler} in the sense of Definition \ref{def:delort}.

\end{theorem}

%\begin{remark} \label{rem:comp}
Theorem \ref{thm:main} is a slight improvement of the convergence result in \cite{SR-02}, which handles initial densities satisfying \eqref{hyp:ini-0}, \eqref{hyp:ini-1} and \eqref{hyp:ini-2-sr}:
\begin{equation}
\label{hyp:ini-2-sr}
 \eps  \|f_\eps^0\|_{L^\infty}\to 0,\quad \text{as }\eps\to 0.
\end{equation}Typically, the assumption \eqref{hyp:ini-2}  allows for initial data such that for some $\beta>1$  \begin{equation*}
\sup_{\eps>0}  \eps^2 |\ln \eps|^\beta  \|f_\eps^0\|_{L^\infty}<+\infty.
\end{equation*}
Thus, Theorem \ref{thm:main} includes initial data that converge to monokinetic data:
\begin{equation*}
\fe^0(x,v)=\rho_0(x)\frac{1}{\eta_\eps^2}F\left(\frac{v-u_\eps(x)}{\eta_\eps}\right),
\end{equation*}
where for instance $u_\eps \in L^2(\R^2)$, $\rho_0\in L^\infty(\R^2)$, $F\in L^1\cap L^\infty(\R^2)$, and $\eps^2\Theta(\eta_\eps^{-2})$ vanishes as $\eps \to 0$.

\medskip

In the case where \eqref{hyp:ini-2} is replaced by the uniform bound
\begin{equation}
\label{hyp:ini-3}
\sup_{\eps>0}   \|f_\eps^0\|_{L^\infty}<+\infty,
\end{equation}
any accumulation point is a true solution of the 2D Euler equation:
\begin{theorem}\label{thm:main-2}
Let $f_\eps^0$ satisfy \eqref{hyp:ini-0}, \eqref{hyp:ini-1} and \eqref{hyp:ini-3}. 
Let $f_\eps$  be the corresponding global weak solution 
to \eqref{syst:VP}.
There exists a subsequence $\eps_n\to 0$ as $n\to +\infty$ such that $(\rho_{\eps_n})_{n\in \N}$ converges to $\rho$ in $C(\R_+,L^2(\R^2)-w)$. Moreover,

\begin{enumerate}

\item $\rho\in L^\infty (\R_+,L^2(\R^2))$;

\item $(E_{\eps_n})_{n\in \N}$ converges to some $E$  in $C(\R_+,L^2_\loc(\R^2))$;

\item For all $t\in \R_+$, $E(t)= (x/|x|^2)\ast  \rho(t)$;

\item $\rho$ is a global weak solution of the 2D Euler equation \eqref{syst:Euler} in the sense of distributions.

\end{enumerate}

\end{theorem}

Besides the already mentioned articles by Golse and Saint-Raymond  \cite{golse-sr} and Saint-Raymond \cite{SR-02}, a wide literature has been devoted 
to the mathematical analysis of the Vlasov equation in the limit of large magnetic or electric field.  Brenier \cite{brenier} derived the Euler equation in a different scaling, 
for smooth and well-prepared data, by means of a different method based on the modulated energy. Various asymptotic regimes for linear or non linear Vlasov equations 
were investigated by 
Fr\'enod and Sonnendr\"ucker \cite{frenod-sonnendrucker-98, frenod-sonnendrucker-99, frenod-sonnendrucker-01},  Golse and Saint-Raymond \cite{golse-sr-2, SR-01}, 
Han-Kwan \cite{han-kwan}, Ghendrih, Hauray and Nouri \cite{ghn}, Hauray and Nouri \cite{hauray-nouri}, and more recently by Bostan, Finot and 
Hauray \cite{bostan-finot-hauray} and by Barr\'e, Chiron, Goudon and Masmoudi \cite{barre-chiron-masmoudi}.  The convergence results in \cite{golse-sr, SR-02} rely on the derivation of an equation for the spatial density
with a good control of the large velocities.
Here, the main ingredient of proof is based on a different weak formulation for the spatial density, following from the ODE satisfied by a suitable combination of 
the characteristics along which the density is essentially constant, see Proposition \ref{prop:ODE-Z}. %Moreover, this combination enables to to filter out the large oscillations in the characteristic system.
This approach actually amounts to focusing on the equation satisfied by the shifted density $f_\eps(t,x-\eps v^\perp,v),$ see Proposition \ref{prop:rho-decale}. 
These so-called gyro-coordinates $(x-v^\perp,v)$ were used in \cite{ghn} 
(see also \cite{hauray-nouri}) for the derivation of a gyrokinetic model from a linear Vlasov equation.
We also mention  that a similar change of variable in the space
variable, in addition to a transformation by rotation in the velocity variable,  has been considered in \cite{frenod-sonnendrucker-01} and in the recent work \cite{bostan-finot-hauray}.

\section{Proof of Theorem \ref{thm:main}}
\label{section:proof-1}

\subsection{Vortex sheet solution of the Euler equation}

We first define the notion of weak solution to the Euler equation \eqref{syst:Euler}, called vortex sheet solution, which is invoked in Theorem \ref{thm:main}.
%This formulation was introduced by Delort \cite{Delort} to generalize the meaning of the product $E^\perp \rho$ when $\rho$ is a measure belonging to $H^{-1}$ (a vortex sheet). 
%We shall use here a slightly different expression given by Shochet \cite{Schochet}. 
\begin{definition}[According to \cite{Delort, Schochet}]\label{def:delort} Let $\rho_0\in \mathcal{M}^+(\RR)\cap H^{-1}(\RR)$ be compactly supported.
We say that $\rho\in C(\R_+, \mathcal{M}^+(\R^2)-w^\ast)\cap L^\infty(\R_+,H^{-1}(\RR))$ is a global weak solution of the Euler equation with initial datum 
$\rho_0$ if we have for all 
$\Phi\in C_c^\infty(\R_+\times \R^2)$
\begin{equation*}\begin{split}
\int_{\R^2}\Phi(t,x)\,d\rho(t,x)=\int_{\R^2}&\Phi(0,x)\,d\rho_0(x)+\int_0^t \int_{\R^2} \partial_s\Phi(s,x)\, d\rho(s,x)\,ds\\
&+\int_0^t \iint_{\R^2\times \R^2}H_{\Phi}(x,y)\,d\rho( s,x)\,d\rho(s,y)\:ds,
\end{split}
\end{equation*}
where $$H_\Phi(x,y)=\frac{1}{2}\frac{(x-y)^\perp}{|x-y|^2}\cdot \left( \nabla \Phi(x)-\nabla \Phi(y)\right).$$ 

\end{definition} For any compactly supported $\rho_0$ in $\mathcal{M}^+(\RR)\cap H^{-1}(\RR)$,  
global existence of a corresponding vortex sheet solution (satisfying a slightly different formulation than the one above)
was established by Delort \cite{Delort}. The formulation of Definition \ref{def:delort}, which has been introduced later by Schochet \cite{Schochet}, is motivated by the observation that 
when $\rho$ is a bounded and integrable map,
\begin{equation}\label{eqi:delort}\langle \text{div}(E^\perp \rho),\Phi\rangle_{\mathcal{D}',\mathcal{D}}=-\iint H_{\Phi}(x,y)\rho(x)\rho(y)\,dx\,dy.\end{equation}
Moreover, $H_\Phi$ is defined and continuous off the diagonal $\Delta=\{(x,x)\:|\:x\in \RR\}$ and bounded on $\RR\times \RR$, since 
 $\|H_\Phi\|_{L^\infty(\RR\times \RR))}\leq \|\Phi\|_{W^{2,\infty}}.$ Hence the expression \eqref{eqi:delort} makes sense for 
 $\rho$ as in Definition \ref{def:delort}, since the atomic support a positive measure in $H^{-1}$ is empty \cite{Delort}.

\subsection{Uniform a priori estimates}

In all the remainder of this section, $f_\eps$ denotes the global weak solution of \eqref{syst:VP} with initial data $\fe^0$ satisfying \eqref{hyp:ini-0}, 
\eqref{hyp:ini-1} and \eqref{hyp:ini-2}. Replacing $\|\fe^0\|_{L^\infty}$ by $\max(1, \|\fe^0\|_{L^\infty})$ if necessary, we will always assume that $$
\|\fe^0\|_{L^\infty}\geq 1.$$
The purpose of this paragraph is to collect a priori estimates and basic properties for $f_\eps$ for later use.
The notation $C$ will stand for a constant independent of $\eps$, changing possibly from a line to another.

\begin{proposition}\label{prop:uni}
We have
\begin{equation*}
\sup_{t\in \R_+}\sup_{\eps>0} \left( \|f_\eps(t)\|_{L^1}+\mathcal{H}(f_\eps(t)) \right)<+\infty,
\end{equation*}
and
\begin{equation*}
\sup_{t\in \R_+}\eps^2 \Theta(   \|f_\eps(t)\|_{L^\infty})\leq \eps^2\Theta(\|f_\eps^0\|)\to 0,\quad \text{as }\eps \to 0.
\end{equation*}

%and finally,
%\begin{equation*}
%\sup_{t\in \R_+}\sup_{\eps>0}\iint_{\RR\times \RR}|\ln|x-y||\rho_\eps(t,x)\rho_\eps(t,y)\,dx\,dy<+\infty.
%\end{equation*}
\end{proposition}

\begin{proof}
This is an immediate consequence of the fact that for \eqref{syst:VP}, the energy and the norms of $f_\eps$ satisfy
\begin{equation*}\label{ineq:uni-bounds}
\forall t\in \R_+,\quad \mathcal{H}(f_\eps(t))\leq  \mathcal{H}_\eps(0),\quad \|f_\eps(t)\|_{L^p}\leq \|f_\eps^0\|_{L^p}.
\end{equation*}
\end{proof}
\begin{proposition}\label{prop:uni-bis}We have for all $t\in \R_+$ and for all $0<\eps<1$
\begin{equation*}\begin{split}\int_{\RR}|x|^2\rho_\eps(t,x)\,dx&\leq C\left(1+\eps^2 \iint_{\RR\times \RR}|v|^2 f_\eps^0(x,v)\,dx\,dv\right)\\&+C\eps^2 \iint_{\RR\times \RR}|v|^2 f_\eps(t,x,v)\,dx\,dv.\end{split}
\end{equation*}

\end{proposition}
\begin{proof} Let $T>0$ and $R_\eps>0$ such that $\text{supp}(f_\eps(t))$ is included in $\overline{B}(0,R_\eps)\times \overline{B}(0,R_\eps)$ on $[0,T]$. We set $\varphi(x,v)=\left( |x|^2+2\eps x\cdot v^\perp \right)\chi(|x|/R_\eps)\chi(|v|/R_\eps)$, where $\chi$ is a smooth cut-off function such that $\chi=1$ on $B(0,1)$ and $\chi=0$ on $B(0,2)^c$. For $t\in [0,T)$, we compute using the weak formulation of \eqref{syst:VP} for the test function $\varphi$,
\begin{equation*}
\begin{split}
&\frac{d}{dt}\iint_{\RR\times \RR}\left( |x+\eps v^\perp|^2-\eps^2|v|^2\right)f_\eps(t,x,v)\,dx\,dv
\\&=
\frac{d}{dt}\iint_{\RR\times \RR}\varphi(x,v)f_\eps(t,x,v)\,dx\,dv\\
&=\iint_{\RR \times \RR}f_\eps(t,x,v) \left(\frac{v}{\eps}\cdot \nabla_x \varphi+\frac{E_\eps}{\eps}\cdot \nabla_v \varphi+\frac{v^\perp}{\eps^2}\cdot \nabla_v \varphi\right)\,dx\,dv\\
&=\iint_{\RR \times \RR}f_\eps(t,x,v) \left(\frac{v}{\eps}\cdot(2x+2 \eps v^\perp)-2 E_\eps \cdot x^\perp-2\frac{v^\perp}{\eps}\cdot x^\perp \right)\,dx\,dv\\
&=-2\int_{\RR }\rho_\eps(t,x) E_\eps(t,x) \cdot x^\perp \,dx.
\end{split}
\end{equation*}On the other hand, in view of the definition of $E_\eps$, we obtain by a classical symmetrization argument
\begin{equation*}
\begin{split}
\int_{\RR }\rho_\eps(t,x) E_\eps(t,x) \cdot x^\perp \,dx&=
\iint_{\RR\times \RR }\rho_\eps(t,x) \rho_\eps(t,y) \frac{x-y}{|x-y|^2} \cdot x^\perp \,dx\,dy\\
&=\frac{1}{2}\iint_{\RR \times \RR }\rho_\eps(t,x) \rho_\eps(t,y) \frac{x-y}{|x-y|^2} \cdot (x^\perp-y^\perp) \,dx\,dy=0.
\end{split}
\end{equation*}
Since $|x|^2\leq 2(|x+\eps v^\perp|^2-\eps^2|v|^2)+4\eps^2|v|^2$, it follows that
\begin{equation*}
\begin{split}
&\iint_{\RR\times \RR}|x|^2f_\eps(t,x,v)\,dx\,dv\\&\leq 2\iint_{\RR\times \RR}\left( |x+\eps v^\perp|^2-\eps^2|v|^2\right)f_\eps^0(x,v)\,dx\,dv+4\eps^2 \iint_{\RR\times \RR}|v|^2f_\eps(t,x,v)\,dx\,dv\\
&\leq C\left(\int_{\RR}|x|^2 \rho_\eps^0(x)\,dx+\eps^2 \iint_{\RR\times \RR}|v|^2f_\eps^0(x,v)\,dx\,dv\right)\\&+C\eps^2 \iint_{\RR\times \RR}|v|^2f_\eps(t,x,v)\,dx\,dv.
\end{split}
\end{equation*}
\end{proof}

\begin{proposition}\label{prop:uni-2}We have\begin{equation*}
\sup_{t\in \R_+}\sup_{\eps>0}\left( \iint_{\RR\times \RR}|v|^2 f_\eps(t,x,v)\,dx\,dv+\int_{\RR}|x|^2\rho_\eps(t,x)\,dx\right)<+\infty,
\end{equation*}and 
\begin{equation*}
\sup_{t\in \R_+}\|\rho_\eps(t)\|_{L^2}\leq C \|f_\eps^0\|_{L^\infty}^{1/2}.
\end{equation*}

Finally, setting $$J_\eps(t,x)=\int_{\R^2}|v|\fe(t,x,v)\,dv,$$we have
\begin{equation*}
\sup_{t\in \R_+}\|J_\eps(t)\|_{L^{4/3}}\leq C \|f_\eps^0\|_{L^\infty}^{1/4} 
\end{equation*}
and
\begin{equation*}
\sup_{t\in \R_+}\sup_{\eps>0}\|J_\eps(t)\|_{L^{1}}<+\infty.
\end{equation*}

\end{proposition}

\begin{proof}
The proof is classical, %see e.g. \cite[Lemma 3.1]{golse-sr}, \cite[Lemma 2.4]{SR-02} or \cite[Proposition 2]{italiens-miot}, 
but we provide  some details for sake of completeness. We omit the dependence on $t$ for simplicity. Setting $$K_\eps=\iint_{\RR\times \RR}|v|^2 f_\eps(x,v)\,dx\,dv,$$ we have the interpolation
inequality (see e.g. \cite[Lemma 3.1]{golse-sr} or \cite[Lemma 2.4]{SR-02})
\begin{equation*}
\|\rho_\eps\|_{L^2}\leq C \|\fe\|_{L^\infty}^{1/2} K_\eps^{1/2}\leq C \|\fe^0\|_{L^\infty}^{1/2} K_\eps^{1/2}.
\end{equation*}
On the other hand, Cauchy-Schwarz inequality and Proposition \ref{prop:uni-bis} yield
\begin{equation*}
\begin{split}
K_\eps&\leq2\mathcal{H}(f_\eps)+\iint_{\RR\times \RR}\ln_+|x-y|\rho_\eps(x)\rho_\eps(y)\,dx\,dy\\
&\leq 2\mathcal{H}(f_\eps^0)+\iint_{\RR\times \RR}\left(|x|+|y|\right)\rho_\eps(x)\rho_\eps(y)\,dx\,dy\\
&\leq C+2\|\rho_\eps\|_{L^1}^{3/2}\left(\int_{\RR}|x|^2\rho_\eps(x)\,dx\right)^{1/2}\\
&\leq C + C\left(1+\eps^2 K_\eps(0)+\eps^2 K_\eps\right)^{1/2}.
\end{split}
\end{equation*}
For the same reasons, we have
\begin{equation*}
\begin{split}
K_\eps(0)&\leq 2 \mathcal{H}(f_\eps^0)+\iint_{\RR\times \RR}\ln_+|x-y|\rho_\eps^0(x)\rho_\eps^0(y)\,dx\,dy\\
&\leq C+ C\|\rho_\eps^0\|_{L^1}^{3/2}\left(\int_{\RR}|x|^2\rho_\eps^0(x)\,dx\right)^{1/2}\leq C
\end{split}\end{equation*} in view of \eqref{hyp:ini-2}.
So we conclude that $K_\eps\leq C$, and by Proposition \ref{prop:uni-bis}, it also follows that $\int_{\RR}|x|^2\rho_\eps(t,x)\,dx\leq C$.

\medskip

Again by interpolation, we have 
\begin{equation*}
\begin{split}
\|J_\eps\|_{L^{4/3}}\leq C \|f_\eps\|_{L^\infty}^{1/4} K_\eps^{3/4}\leq C \|f_\eps^0\|_{L^\infty}^{1/4} K_\eps^{3/4},
\end{split}
\end{equation*}
and by Cauchy-Schwarz inequality, we obtain
\begin{equation*}
 \|J_\eps\|_{L^1}\leq C\|f_\eps\|_{L^1}^{1/2}K_\eps^{1/2}\leq C\|f_\eps^0\|_{L^1}^{1/2}K_\eps^{1/2},
\end{equation*}
so the conclusion follows.

\end{proof}

To conclude this paragraph, we introduce a smooth, positive function $\tilde{\rho}_\eps$, compactly supported in $B(0,1)$, such that
\begin{equation}
\label{bound:stationary}
\int_{\RR}\tilde{\rho}_\eps(x)\,dx=\int_{\RR}{\rho}_\eps(x)\,dx,\quad \sup_{\eps>0}\|\tilde{\rho}_\eps\|_{L^\infty}<+\infty\end{equation}
and we set $$\tilde{E}_\eps(x)= \int_{\RR}\frac{x-y}{|x-y|^2}\tilde{\rho}_{\eps}(y)\,dy.$$Since $\int (\rho_\eps(t)-\tilde{\rho}_\eps)=0$ and 
$\rho_\eps(t)-\tilde{\rho}_\eps$ is compactly supported, it is well-known that $E_\eps(t)-\tilde{E}_\eps$ belongs to $L^2(\RR)$, see e.g. \cite[Proposition 3.3]{majda-bertozzi}. 
In addition,
\begin{proposition}
We have
$$\sup_{t\in \R_+}\sup_{\eps>0}\|E_\eps(t)-\tilde{E}_\eps\|_{L^2}<+\infty.$$
\label{prop:e-barre}
\end{proposition}

\begin{proof}The computations below are quite standard and we perform them for sake of completeness. We first integrate by parts, using that $E_\eps(t)-\tilde{E}_\eps=
2\pi \nabla G\ast (\rho_\eps(t)-\tilde{\rho}_\eps)$, with $G$ 
the fundamental solution of the Laplacian in $\RR$. Then we expand, which yields
\begin{equation*}\begin{split}
\|E_\eps(t)-\tilde{E}_\eps\|_{L^2}^2
&=-2\pi\iint_{\R^2\times \RR}
\ln|x-y|\left(\rho_\eps-\tilde{\rho}_\eps\right)(t,x)\left(\rho_\eps-\tilde{\rho}_\eps\right)(t,y)\,dx\,dy\\
&\leq -2\pi\iint_{\RR\times \RR}\ln|x-y|\rho_\eps(t,x)\rho_\eps(t,y)\,dx\,dy\\
&-2\pi\iint_{B(0,1)^2}\ln_-|x-y|\tilde{\rho}_\eps(x)\tilde{\rho}_\eps(y)\,dx\,dy\\
&+4\pi\iint_{\RR\times B(0,1)}\ln_+|x-y|\rho_\eps(t,x)\tilde{\rho}_\eps(y)\,dx\,dy.\end{split}
\end{equation*}
Then we use Proposition \ref{prop:uni-2} and \eqref{bound:stationary} to infer that
\begin{equation*}\begin{split}
&\|E_\eps(t)-\tilde{E}_\eps\|_{L^2}^2\\
&\leq C\Big(\mathcal{H}(f_\eps(t))+\|\tilde{\rho}_\eps\|_{L^\infty}^2+\iint_{\RR\times B(0,1)}(|x|+|y|)\rho_\eps(t,x)\tilde{\rho}_\eps(y)\,dx\,dy\Big)\\
&\leq C\Big(\mathcal{H}(f_\eps(t))+\|\tilde{\rho}_\eps\|_{L^\infty}^2+\|\tilde{\rho}_\eps\|_{L^\infty}\|\rho_\eps(t)\|_{L^1}^{1/2}\left(\int_{\RR}|x|^2\rho_\eps(t,x)\,dx\right)^{1/2}\Big)\\
&+C\|\tilde{\rho}_\eps\|_{L^\infty}\|\rho_\eps(t)\|_{L^1}\\
&\leq C.
\end{split}
\end{equation*}

\end{proof}

\begin{proposition}\label{prop:uni-3}
We have $E_\eps-\tilde{E}_\eps\in L^\infty(\R_+,H^1(\RR))$ and
\begin{equation*}
\sup_{t\in \R_+}\|E_\eps(t)-\tilde{E}_\eps\|_{H^1(\RR)}\leq C  \|f_\eps^0\|_{L^\infty}^{1/2}.
\end{equation*}
In particular, for all $q\geq 2$ we have
\begin{equation*}
\sup_{t\in \R_+}\|E_\eps(t)-\tilde{E}_\eps\|_{L^q}\leq C\sqrt{q}  \|f_\eps^0\|_{L^\infty}^{1/2}.
\end{equation*}

\end{proposition}

\begin{proof}
On the one hand, $\|E_\eps(t)-\tilde{E}_\eps\|_{L^2}\leq C$ by virtue of Proposition \ref{prop:e-barre}. On the other hand, standard elliptic regularity theory yields a constant $C>0$  such that 
\begin{equation*}\begin{split}\|\nabla(E_\eps(t)-\tilde{E}_\eps)\|_{L^2}\leq 
 C\|\rho_\eps(t)-\tilde{\rho}_\eps\|_{L^2}
&\leq C \|f_\eps^0\|_{L^\infty}^{1/2},\end{split}\end{equation*}
where we have used Proposition \ref{prop:uni-2} and \eqref{bound:stationary} in the last inequality.

Finaly, the second statement follows from the Sobolev embedding theorem $H^1(\RR)\subset L^q(\RR)$ for all $q\geq 2$, with the dependence of the  constant with respect to $q$ given in, e.g., \cite[Paragraph 8.5, p. 206]{lieb-loss}.
\end{proof}

\subsection{Lagrangian trajectories and weak formulation}

We introduce the field $$b_\eps(t,x,v)=\left(\frac{v}{\eps}, \frac{E_\eps(t,x)}{\eps}+\frac{v^\perp}{\eps^2}\right),$$which satisfies
$$\frac{b_\eps}{1+|x|+|v|}\in L^1_\loc(\R_+, L^1(\R^2\times \R^2))+L^1_\loc(\R_+, L^\infty(\R^2\times \R^2)),$$
see e.g. \cite[Proposition 6.2]{bohun-bouchut-crippa}. 
Moreover, by Proposition \ref{prop:uni-3}, we have\footnote{$Db$ denotes the  differential matrix of $b$ with respect to $x$ and $v$.}
$$D b_\eps \in L^1_\loc(\R_+, L^2_\loc(\R^2\times \R^2)).$$
Therefore, the DiPerna and Lions theory \cite{dip-lions} applies, providing a unique Lagrangian flow associated to $b_\eps$, 
which we denote by $(X_\eps,V_\eps)$. We refer to the recent survey \cite{ambrosio-trevisan} or to \cite{bohun-bouchut-crippa}, which handles specifically the Vlasov-Poisson case.
In particular, for almost every $(x,v)\in \R^2\times \R^2$, $t\mapsto (X_\eps(t,x,v), V_\eps(t,x,v)$ is an absolutely continuous map which satisfies 
\begin{equation}\label{syst:ODE}
\begin{cases}
\displaystyle X_\eps(t,x,v)=x+\frac{1}{\eps}\int_0^t V_\eps(s,x,v)\,ds\\
\displaystyle V_\eps(t,x,v)=v+\frac{1}{\eps^2}\int_0^t \left(V_\eps^\perp(s,x,v)+\eps E_\eps(s,X_\eps(s,x,v))\right)\,ds.
\end{cases}
\end{equation} Moreover, the solution $f_\eps$ is the push-forward\footnote{In view of the support properties of $f_\eps$, this means here that for all 
$\varphi\in L^1_{\loc}(\RR \times \RR)$, we have $\iint f_\eps(t,x,v)\varphi(x,v)\,dx \,dv=\iint f_\eps^0(x,v)\varphi(X_\eps(t,x,v),V_\eps(t,x,v))\,dx\,dv$.} of the initial density $f_\eps^0$ by the flow,
\begin{equation}
\label{push-forward}
\fe(t)=(X_\eps(t), V_\eps(t))_\#f_\eps^0.\end{equation}Recalling that $\rho_\eps$ belongs to $L^\infty_\loc(\R_+,L^\infty(\RR))$ for all $0<\eps<1$, we infer that $E_\eps$ satisfies
\begin{equation*}\begin{split}\forall T>0,\quad &\sup_{t\in [0,T]}\|E_\eps(t)\|_{L^\infty}\leq C({\eps,T}),\\
&\sup_{t\in [0,T]}|E_\eps(t,x)-E_\eps(t,y)|\leq 
C({\eps,T})|x-y|(1+|\ln |x-y||)\end{split}\end{equation*} (see e.g. \cite[Lemma 4]{LP}). Thus it turns out that 
for all $(x,v)\in \RR \times \RR$ the map $t\mapsto (X_\eps(t,x,v),V_\eps(t,x,v))$ belongs to $W^{1,\infty}(\R_+,\RR\times \RR)$ and is the unique solution to the ODE \eqref{syst:ODE}.

\medskip

We define then the following combination of the characteristics:
$$Z_\eps(t,x,v)=X_\eps(t,x,v)+\eps V_\eps^\perp(t,x,v).$$ 

\begin{proposition}\label{prop:ODE-Z}
For all $(x,v)\in \R^2\times \R^2$, the map $t\mapsto Z_\eps(t,x,v)$ belongs to $W^{1,\infty}(\R_+,\RR)$ and it satisfies
$$\dot{Z}_\eps(t,x,v)=E_\eps^\perp(t,X_\eps(t,x,v)),\quad \text{for a.e. }t\in \R_+.$$

\end{proposition}

\begin{proof}We have for a.e. $t\in \R_+$
\begin{equation*}
\begin{split}
\dot{Z}_\eps(t)&=\frac{V_\eps(t)}{\eps}+\eps\left(\frac{V_\eps^\perp(t)+\eps E_\eps(t,X_\eps(t))}{\eps^2}\right)^\perp=E_\eps^\perp(t,X_\eps(t)).
\end{split}
\end{equation*}
\end{proof}
We can now derive a weak formulation for the spatial density.
\begin{proposition}\label{prop:weak-formulation-1}
Let $\Phi \in C_c^\infty(\R_+\times \R^2)$. We have 
\begin{equation*}
\begin{split}
%\frac{d}{dt}\iint_{\R^2\times \R^2} \fe(x,v)\Phi(x)\,dx
\int_{\R^2} \rho_\eps(t,x)\Phi(t,x)\,dx&-\int_{\R^2} \rho_\eps^0(x)\Phi(0,x)\,dx=\int_0^t\int_{\RR}\partial_s \Phi(s,x)\rho_\eps(s,x)\,dx\,ds \\
&+\int_0^t \iint_{\R^2\times \RR}H_{\Phi(s,\cdot)}(x,y)\rho_\eps(s,x)\rho_\eps(s,y)\,dx\,dy\,ds+{R}_\eps(t),
\end{split}
\end{equation*}
where ${R}_\eps$ converges to zero locally uniformly on $\R_+$ as $\eps\to 0$. More precisely,
\begin{equation*}
 |R_\eps(t)|\leq C(1+t) \|\Phi\|_{L^\infty(W^{2,\infty})}\left(\eps^2 \Theta(\|f_\eps^0\|_{L^\infty})\right)^{1/2}.
\end{equation*}

\end{proposition}

\begin{proof} Thanks to \eqref{push-forward}, we may write
\begin{equation*}
\begin{split}
%\frac{d}{dt}\iint_{\R^2\times \R^2} \fe(x,v)\Phi(x)\,dx
\iint_{\R^2\times \R^2} \fe(t,x,v)\Phi(t,x)\,dx\,dv
&=\iint_{\R^2\times \R^2} \fe^0(x,v)\Phi(t,Z_\eps(t,x,v))\,dx\,dv+R_{\eps,1}(t),
\end{split}
\end{equation*}
where
\begin{equation*}
R_{\eps,1}(t)=\iint_{\R^2\times \R^2} \fe^0(x,v)\left( \Phi(t,X_\eps(t,x,v))-\Phi(t,Z_\eps(t,x,v))
\right)\,dx\,dv.
\end{equation*}

On the one hand, we have by the mean-value theorem
\begin{equation*}
|R_{\eps,1}(t)|\leq \|D \Phi(t)\|_{L^\infty}\eps \iint_{\R^2\times \R^2} \fe^0(x,v)|V_\eps(t,x,v)|\,dx\,dv
\end{equation*}
hence using \eqref{push-forward} and Proposition \ref{prop:uni} we get
\begin{equation*}%\label{ineq:r1}
\sup_{t\in \R_+} |R_{\eps,1}(t)|\leq C \,\eps\|D \Phi(t)\|_{L^\infty}.
\end{equation*}

On the other hand, Proposition \ref{prop:ODE-Z} implies that for all $(x,v)\in \RR\times \RR$, the map $t\mapsto \Phi(t,Z_\eps(t,x,v))$ belongs to $W^{1,\infty}(\R_+)$ therefore
\begin{equation*}
\begin{split}
&\iint_{\R^2\times \R^2} \fe^0(x,v)\Phi(t,Z_\eps(t,x,v))\,dx\,dv\\&= \iint_{\R^2\times \R^2} \fe^0(x,v)\Phi(0,Z_\eps(0,x,v))\,dx\,dv\\
&+\iint_{\R^2\times \R^2} \fe^0(x,v)\int_0^t \frac{d}{ds}\Phi(s,Z_\eps(s,x,v))\,ds\:dx\,dv\\
&=\iint_{\R^2\times \R^2} \fe^0(x,v)\Phi(0,x+\eps v^\perp)\,dx\,dv\\
&+\iint_{\R^2\times \R^2} \fe^0(x,v)\int_0^t\partial_s \Phi(s,Z_\eps(s,x,v))\,ds\:dx\,dv\\
&+\iint_{\R^2\times \R^2} \fe^0(x,v)\int_0^t\nabla \Phi(s,Z_\eps(s,x,v))\cdot E_\eps^\perp(s,X_\eps(s,x,v))\,ds\:dx\,dv.\end{split}\end{equation*}
Using again \eqref{push-forward}, we obtain
\begin{equation*}
\begin{split}
&\iint_{\R^2\times \R^2} \fe^0(x,v)\Phi(t,Z_\eps(t,x,v))\,dx\,dv\\
&=\iint_{\R^2\times \R^2} \fe^0(x,v)\Phi(0,x+\eps v^\perp)\,dx\,dv\\
&+\int_0^t \iint_{\R^2\times \R^2} \fe(s,x,v)\partial_s \Phi(s,x+\eps v^\perp)\,ds\:dx\,dv\\
&+\int_0^t \, \iint_{\R^2\times \R^2} \fe(s,x,v)\nabla \Phi(s,x+\eps v^\perp )\cdot E_\eps^\perp(s,x)\:dx\,dv\:ds.
\end{split}
\end{equation*}
Therefore, we have
\begin{equation*}
\begin{split}
&\iint_{\R^2\times \R^2} \fe^0(x,v)\Phi(t,Z_\eps(t,x,v))\,dx\,dv\\
&=\iint_{\R^2\times \R^2} \fe^0(x,v)\Phi(0,x)\,dx\,dv
+\int_0^t \iint_{\R^2\times \R^2} \fe(s,x,v)\partial_s \Phi(s,x)\,ds\:dx\,dv\\
&+\int_0^t \, \iint_{\R^2\times \R^2} \fe(s,x,v)\nabla \Phi(s,x )\cdot E_\eps^\perp(s,x)\:dx\,dv\:ds+\sum_{k=2}^5R_{\eps,k}(t),
\end{split}
\end{equation*}
where
\begin{equation*}
\begin{split}
R_{\eps,2}(t)&=\iint_{\R^2\times \R^2} \fe^0(x,v)\left(\Phi(0,x+\eps v^\perp)-\Phi(0,x)\right)\,dx\,dv,\\
R_{\eps,3}(t)&=\int_0^t \iint_{\R^2\times \R^2} \fe(s,x,v)\left(\partial_s \Phi(s,x+\eps v^\perp)-\partial_s \Phi(s,x)\right)\,dx\,dv\,ds,\\
R_{\eps,4}(t)&=\int_0^t \, \iint_{\R^2\times \R^2} \fe(s,x,v)\left(\nabla \Phi(s,x+\eps v^\perp)-\nabla \Phi(s,x )\right)\cdot (E_\eps^\perp(s,x)-\tilde{E}_\eps^\perp(x))\,dx\,dv\,ds,\\
R_{\eps,5}(t)&=\int_0^t \, \iint_{\R^2\times \R^2} \fe(s,x,v)\left(\nabla \Phi(s,x+\eps v^\perp)-\nabla \Phi(s,x )\right)\cdot  \tilde{E}_\eps^\perp(x)\,dx\,dv\,ds.
\end{split}
\end{equation*}

On the one hand, inserting the definition of $E_\eps$ and symmetrizing as in \cite{Schochet}, we get 
\begin{equation*}
 \int_{\R^2} \rho_\eps(s,x)\nabla \Phi(s,x )\cdot E_\eps^\perp(s,x)\:dx=\iint_{\RR\times\RR} H_{\Phi(s,\cdot)}(x,y)\rho_\eps(s,x)\rho_\eps(s,y)\,dx\,dy.
\end{equation*}

On the other hand, as before, we  obtain
\begin{equation*}
 |R_{\eps,2}(t)|\leq C \eps  \|\nabla  \Phi(0)\|_{L^\infty}.
\end{equation*}

Besides, Proposition \ref{prop:uni-3} yields
\begin{equation*}\begin{split}
 |R_{\eps,3}(t)|&\leq C \eps t  \|  D\partial_s \Phi\|_{L^\infty(L^\infty)}\sup_{s\in \R_+}\iint_{\RR \times \RR}|v|f_\eps(s,x,v)\,dx\,dv\leq Ct\eps\|D^2\Phi\|_{L^\infty(L^\infty)}.
 \end{split}
\end{equation*}

Next, we infer from the mean-value theorem, H\"older inequality and Proposition \ref{prop:uni-3} that
\begin{equation*}
\begin{split}
|R_{\eps,4}(t)|&\leq \eps \, \|D^2\Phi\|_{L^\infty(L^\infty)}\int_0^t \iint_{\R^2\times \R^2}f_\eps(s,x,v)|v||E_\eps(s,x)-\tilde{E}_\eps(x)|\,dx\,dv\,ds\\
&\leq Ct \eps \, \|D^2\Phi\|_{L^\infty(L^\infty)}\sup_{s\in [0,t]}\left(\|E_\eps(s)-\tilde{E}_\eps\|_{L^q} \|J_\eps(s)\|_{L^{q'}}\right)\\
&\leq  Ct \eps \, \|D^2\Phi\|_{L^\infty(L^\infty)}\sqrt{q}\|f_\eps^0\|_{L^\infty}^{1/2}\sup_{s\in [0,t]}\|J_\eps(s)\|_{L^{q'}},
\end{split}
\end{equation*}
where $q'$ is the conjugate exponent of $q$, and where $q\geq 4$ will be chosen later. Since $q'\in(1,4/3]$, we have
$$\|J_\eps(s)\|_{L^{q'}}\leq \|J_\eps(s)\|_{L^1}^{1-\frac{4}{q}}\|J_\eps(s)\|_{L^{4/3}}^{\frac{4}{q}},$$
thus Proposition \ref{prop:uni-2} yields
\begin{equation*}
\begin{split}
|R_{\eps,4}(t)|
&\leq  Ct \eps \, \|D^2\Phi\|_{L^\infty(L^\infty)}\sqrt{q}\|f_\eps^0\|_{L^\infty}^{\frac{1}{2}+\frac{1}{q}}.
\end{split}
\end{equation*}

Finally, we set 
$$q=\max(4,\ln (\|\fe^0\|_{L^\infty})),$$
so that
\begin{equation*}
\begin{split}
|R_{\eps,4}(t)|
&\leq  Ct \eps \, \|D^2\Phi\|_{L^\infty(L^\infty)}\Theta\left(\|f_\eps^0\|_{L^\infty}\right)^{1/2}.
\end{split}
\end{equation*}

We turn to the last term. We infer from \eqref{bound:stationary} and from classical potential estimates, see e.g. \cite{majda-bertozzi}, that
\begin{equation*}
 \sup_{\eps>0}\|\tilde{E}_\eps\|_{L^\infty}\leq C\sup_{\eps>0}\|\tilde{\rho}_\eps\|_{L^1}^{1/2}\|\tilde{\rho}_\eps\|_{L^\infty}^{1/2}\leq C,
\end{equation*}
therefore
\begin{equation*}
 \begin{split}
|R_{\eps,5}(t)|&\leq C t \eps\, \|D^2\Phi\|_{L^\infty(L^\infty)} \|\tilde{E}_\eps\|_{L^\infty}\sup_{s\in [0,t]}
\int_{\RR\times \RR}|v|f_\eps(s,x,v)\,dx\,dv\\&\leq Ct\eps\|D^2\Phi\|_{L^\infty(L^\infty)}.  
 \end{split}
\end{equation*}

Gathering the previous bounds and recalling that $\Theta(\|f_\eps^0\|_{L^\infty})\geq 1$, we obtain the desired estimate.

\end{proof}

\subsection{Passing to the limit}\label{subsec:limit}We establish a property of uniform equicontinuity with respect to time for the spatial densities.
\begin{lemma}\label{lemma:equicontinuity}There exists $K_0>0$ such that for all $s,t\in \R_+$, 
 \begin{equation*}
  \|\rho_\eps(t)-\rho_\eps(s)\|_{W^{-2,1}(\RR)}\leq K_0\left(|t-s|+(1+t+s)\eps\Theta(\|f_\eps^0\|_{L^\infty})^{1/2}\right).
 \end{equation*}

\end{lemma}

\begin{proof}This is a simple consequence of Proposition \ref{prop:weak-formulation-1} and of the estimate $|H_{\Phi}(x,y)|\leq \|\Phi\|_{W^{2,\infty}}$.

\end{proof}
We are now in position to complete the proof of Theorem \ref{thm:main}. A straightforward adaptation of Ascoli's theorem yields:
\begin{lemma}Let $T>0$. Let $(F,d)$ be a complete metric space.
 Let $(f_n)_{n\in \N}$ be a family of $C([0,T], F)$ such that 
 \begin{enumerate}
  \item For all $t\in [0,T]$, $(f_n(t))_{n\in \N}$ is relatively compact in $F$;
  \item There exists $C>0$ and a sequence $r_n\to 0$ as $n\to +\infty$ such that for all $t,s\in [0,T]$, for all $n\in \mathbb{N}$, 
  $|f_n(t)-f_n(s)|\leq C|t-s|+r_n$.
 \end{enumerate}Then the family $(f_n)_{n\in \N}$ is relatively compact in $C([0,T],F)$.
\end{lemma}
Using the fact that $(\rho_\eps)_{\eps>0}$ is uniformly bounded in $L^\infty(\R_+, \mathcal{M}^+(\RR))$ and recalling Lemma \ref{lemma:equicontinuity}, 
we can apply this Lemma to $F=W^{-2,1}(\RR)$ and we can mimick the proof of Lemma 3.2 in \cite[Lemma 3.2]{Schochet} to show that there exists $\eps_n\to 0$ as $n\to +\infty$ such that $(\rho_{\eps_n})_{n\in \N}$ converges to some $\rho$ in $C(\R_+,\mathcal{M}^+(\RR)-w^\ast)$. By Proposition \ref{prop:e-barre}, $(\rho_{\eps_n})_{n\in \N}$ is bounded in $L^\infty(\R_+,H^{-1}(\RR))$. It was proved 
in \cite{Delort} (see also \cite{majda-93,Schochet}) that this implies that
the non-linear term $$\int \iint_{} H_{\Phi}(x,y)\rho_{\eps_n}(s,x)\rho_{\eps_n}(s,y)\,dx\,dy\,ds$$ converges to $$\int \iint_{} H_{\Phi}(x,y)\rho(s,x)\rho(s,y)\,dx\,dy\,ds$$ as $n\to +\infty$
for all test function $\Phi$. On the other hand, all linear terms appearing in the formulation given by Proposition \ref{prop:weak-formulation-1} pass to the limit. This means that $\rho$ satisfies the conclusion of Theorem \ref{thm:main}.

\subsection{Alternative proof of Theorem \ref{thm:main} without Lagrangian trajectories}The purpose of this paragraph is to propose 
another proof of Theorem \ref{thm:main}, for smooth solutions, that does not rely on the characteristics. Here, we assume that the initial data $f_\eps^0$ satisfy the assumptions of Theorem \ref{thm:main} 
and that moreover $$f_\eps^0\in C^{1,\alpha}
(\RR\times \RR)$$ for some $\alpha\in (0,1)$. The corresponding solution to \eqref{syst:VP} then belongs to $C^1(\R_+\times \RR\times \RR)$. 

As in \cite{ghn}, we consider the microscopic and macroscopic densities in the gyro-coordinates:
$$\overline{\fe}(t,x,v)=\fe(t,x-\eps v^\perp,v),\quad \overline{\rho}_\eps(t,x)=\int_{\RR} \overline{\fe}(t,x,v)\,dv. $$
\begin{proposition}\label{prop:rho-decale}We have
\begin{equation*}\label{eq:decale}
\partial_t \overline{\fe}+E_\eps^\perp(t,x-\eps v^\perp)\cdot  \nabla_x \overline{\fe}
+\left(\frac{v^\perp}{\eps^2}+\frac{E_\eps(t,x-\eps v^\perp)}{\eps}\right)\cdot \nabla_v \overline{\fe}=0,
\end{equation*}
and
\begin{equation*}\label{eq:rho-decale}
\partial_t \overline{\rho}_\eps+\nabla_x \cdot \left( \int_{\RR} E_\eps^\perp(t,x-\eps v^\perp) \overline{\fe} \,dv\right)=0.
\end{equation*}
\end{proposition}
 
 \begin{proof} 
We compute
\begin{equation*}
\begin{split}
\partial_t \overline{\fe}(t,x,v)&=\partial_t \fe(t,x-\eps v^\perp,v),\quad \nabla_x \overline{\fe}(t,x,v)=\nabla_x \fe(t,x-\eps v^\perp,v),\\
\nabla_v \overline{\fe}(t,x,v)&=(\eps \nabla_x^\perp +\nabla_v)\fe(t,x-\eps v^\perp,v),
\end{split}
\end{equation*}
then
\begin{equation*}
\begin{split}
&\left(\frac{v^\perp}{\eps^2}+\frac{E_\eps(t,x-\eps v^{\perp})}{\eps}\right)\cdot \nabla_v \fe(t,x-\eps v^\perp,v)\\&
=\left(\frac{v^\perp}{\eps^2}+\frac{E_\eps(t,x-\eps v^\perp)}{\eps}\right)\cdot \left(\nabla_v-\eps \nabla_x^\perp\right) \overline{\fe}(t,x,v)\\
&=-\frac{v}{\eps}\cdot \nabla_x \overline{\fe}(t,x,v)+E_\eps^\perp(t,x-\eps v^\perp)\cdot  \nabla_x \overline{\fe}(t,x,v)
+\left(\frac{v^\perp}{\eps^2}+\frac{E_\eps(t,x-\eps v^\perp)}{\eps}\right)\cdot \nabla_v \overline{\fe}(t,x,v).
\end{split}
\end{equation*}
Therefore $\overline{\fe}$ satisfies the first equation in Proposition \ref{prop:rho-decale}. Next, we integrate with respect to $v$ and we observe that
\begin{equation*}
\begin{split}
\int_{\RR} v^\perp \cdot \nabla_v \overline{\fe}\,dv&=-\int_{\RR}\nabla_v \cdot v^\perp\, \overline{\fe}\,dv=0,
\\
 \int_{\RR} E_\eps^\perp(x-\eps v^\perp) \cdot \nabla_v \overline{\fe}\,dv&=-\int_{\RR}\nabla_v \cdot [ E_\eps(x-\eps v^\perp)]\, \overline{\fe}\,dv
=\eps \int_{\RR} \text{curl} (E_\eps)(x-\eps v^\perp)\, \overline{\fe}\,dv,
\end{split}
\end{equation*}
where $\text{curl} (G)=\partial_2G_1-\partial_1 G_2.$ Similarly,
\begin{equation*}
\begin{split}
\int_{\RR} E_\eps^\perp(x-\eps v^\perp) \cdot \nabla_x \overline{\fe}\,dv&=
\nabla_x\cdot \left(\int_{\RR} E_\eps^\perp(x-\eps v^\perp)\overline{\fe}\,dv\right)-\int_{\RR} \text{curl}(E_\eps)(x-\eps v^\perp)\,\overline{\fe}\,dv.
\end{split}
\end{equation*}
Now, since $E_\eps$ is a gradient, we have $\text{curl}(E_\eps)=0$, hence the second equation of Proposition \ref{prop:rho-decale} follows.

\end{proof}

We now establish Theorem \ref{thm:main}. The same arguments as the ones of Subsection \ref{subsec:limit} yield a subsequence such that $(\rho_{\eps_n})_{n\in \N}$ 
converges to $\rho$ in $C(\R_+, \mathcal{M}^+(\RR)-w^\ast)$ as $n\to +\infty$. 
Let $\Phi \in C_c^\infty(\R_+\times\R^2)$. Using Proposition \ref{prop:rho-decale} and the fact that the Jacobian of $x\mapsto x+\eps v^\perp$ is one for any fixed $v$,  
we obtain
\begin{equation*}
\begin{split}
&\frac{d}{dt} \int_{\RR}\overline{\rho}_{\eps_n}(t,x)\Phi(t,x)\,dx\\
&=\int_{\RR}\overline{\rho}_{\eps_n}(t,x)\dt \Phi(t,x)\,dx
+\int_{\RR}\nabla \Phi(t,x)\cdot\left(\int_{\RR} E_{\eps_n}^\perp(t,x-\eps_n v^\perp) \overline{f}_{\eps_n}(t,x,v) \,dv\right) \,dx\\
&=\int_{\RR}\left(\int_{\RR}{f}_{\eps_n}(t,x-\eps_n v^{\perp},v) 
\left[\dt \Phi(t,x)+\nabla \Phi(t,x)\cdot E_{\eps_n}^\perp(t,x-\eps_n v^\perp)\right]\,dx\right)\,dv\\
&=\int_{\RR}\left(\int_{\RR}{f}_{\eps_n}(t,x,v) \left[\dt \Phi(t,x+\eps_n v^\perp) +\nabla 
\Phi(t,x+\eps_n v^\perp )\cdot E_{\eps_n}^\perp(t,x)\right]\,dx\right)\,dv.
\end{split}
\end{equation*}
Writing finally
\begin{equation*}
\begin{split}
&\int_{\RR}\left(\int_{\RR}f_{\eps_n}(t,x,v) \left[\dt \Phi(t,x+\eps_n v^\perp) +\nabla 
\Phi(t,x+\eps_n v^\perp )\cdot E_{\eps_n}^\perp(t,x)\right]\,dx\right)\,dv
\\&=\int_{\RR}\rho_{\eps_n}(t,x)\left[\partial_t \Phi(t,x)+ \nabla \Phi(t,x)\cdot E_{\eps_n}^\perp(t,x)\right]\,dx
\\&+\iint_{\RR \times \RR}f_{\eps_n}(t,x,v) \left[ \dt \Phi(t,x+\eps_n v^\perp)-\dt \Phi(t,x)\right]\,dx\,dv\\
&+\iint_{\RR \times \RR}f_{\eps_n}(t,x,v) \left[ \nabla \Phi(t,x+\eps_n v^\perp )-\nabla \Phi(t,x)\right]
\cdot E_{\eps_n}^\perp(t,x)\,dx\,dv,
\end{split}
\end{equation*}
we conclude as in the previous section.

\section{Proof of Theorem \ref{thm:main-2}}

In this section we adapt the proof of Theorem \ref{thm:main} to the case of initial data satisfying the assumptions of Theorem \ref{thm:main-2}. We have
\begin{equation}
\sup_{t\in \R_+}\sup_{\eps>0} \|\fe(t)\|_{L^\infty}<\infty,
\end{equation}hence it follows from Propositions \ref{prop:uni-2} and \ref{prop:uni-3} that
\begin{equation}\label{ineq:ajout}
\sup_{t\in \R_+}\sup_{\eps>0} \|\rho_\eps(t)\|_{L^2(\RR)}<\infty,\quad
\sup_{t\in \R_+}\sup_{\eps>0} \|E_\eps(t)-\tilde{E}_\eps\|_{H^1(\RR)}<\infty.\end{equation}Exactly as in Subsection \ref{subsec:limit}, the 
family $(\rho_\eps)_{\eps>0}$ is relatively compact in
$C(\R_+,\mathcal{M}^+(\RR)-w^\ast)$. Moreover, $(\rho_\eps(t))_{\eps>0}$ is weakly relatively compact in $L^2(\RR)$ for all $t\geq 0$. It follows that for some subsequence $\eps_n \to 0$,
$(\rho_{\eps_n})_{n\in \N}$ converges to some $\rho$ in $C(\R_+,L^2(\RR)-w)$ and in $C(\R_+,\mathcal{M}^+(\RR)-w^\ast)$ as $n\to +\infty$. 
Let $E=( x/|x|^2)\ast \rho$, so that $E$ belongs to $L^\infty_{\loc}(\R_+,L^1+L^2(\RR))$.  Decomposing $$ \frac{x}{|x|^2}=\frac{x}{|x|^2}\chi_\delta+ 
\frac{x}{|x|^2}(1-\chi_\delta),$$ 
with $\chi_\delta$ a cut-off function supported in $B(0,2\delta)$ with value $1$ on $ B(0,\delta)$, we see immediately that 
$$ \frac{x}{|x|^2}(1-\chi_\delta)\ast \rho_{\eps_n}\to  \frac{x}{|x|^2}(1-\chi_\delta)\ast \rho \quad \text{locally uniformly on $\R_+\times \RR$ as $n\to +\infty$},$$ while 
$$\left\|\left( \frac{x}{|x|^2}\chi_\delta\right) \ast \rho_{\eps_n}(t)\right\|_{L^2}\leq C\delta \|\rho_{\eps_n}(t)\|_{L^2}\leq C\delta.$$ So we conclude that 
$(E_{\eps_n})_{n\in \N}$ converges to $E$ in 
$C(\R_+, L^2_{\loc}(\RR))$. This implies that
$(E_{\eps_n}^\perp\rho_{\eps_n})_{n\in \N}$ converges to $E^\perp \rho$
in the sense of distributions on $\R_+\times \RR$. Therefore, all terms pass to the limit in Proposition \ref{prop:weak-formulation-1}, and $\rho$ satisfies 
\eqref{syst:Euler} in the sense of distributions. This concludes the proof.

\medskip

\noindent \textbf{Acknowledgments} {The author  is partially supported by the French ANR projects SchEq ANR-12-JS-0005-01 and GEODISP ANR-12-BS01-0015-01.}


\begin{thebibliography}{99}




\bibitem{ambrosio-trevisan} L. Ambrosio and D. Trevisan, \emph{Lecture notes on the DiPerna-Lions theory on transport equations in abstract measure spaces} (2015), preprint.


\bibitem{Arsenev}  A. A. Arsenʹev, \emph{Existence in the large of a weak solution of Vlasov's system of equations} (Russian), \u{Z}. Vy\u{c}isl. Mat. i Mat. Fiz. \textbf{15} 
(1975), 136--147, 276.

\bibitem{barre-chiron-masmoudi} J. Barr\'e, D. Chiron, T. Goudon and N. Masmoudi, \emph{From Vlasov-Poisson and Vlasov-Poisson-Fokker-Planck Systems to Incompressible Euler Equations: the case with finite charge}, preprint
arXiv:1502.07890, 2015.

\bibitem{bohun-bouchut-crippa} A. Bohun, F. Bouchut and G. Crippa, \emph{Lagrangian solutions to the Vlasov-Poisson system with $L^1$ density}, preprint  arXiv:1412.6358, 2014.

\bibitem{bostan-finot-hauray} M. Bostan, A. Finot and M. Hauray, \emph{The effective Vlasov-Poisson system for strongly magnetized plasmas}, preprint arXiv:1511.00169, 2015.

\bibitem{brenier} Y. Brenier, \emph{Convergence of the Vlasov-Poisson system to the incompressible Euler
equations}, Comm. Partial Differential Equations \textbf{25} (2000), 737--754.



%\bibitem{italiens-miot} S. Caprino, C. Marchioro, E. Miot and M. Pulvirenti, \emph{On the attractive plasma-charge model in 2-D},  Comm. Partial Differential Equations \textbf{37} (2012), no. 7, 1237--1272.


\bibitem{Delort} J.-M. Delort, \emph{Existence de nappes de tourbillon en dimension deux} (French) [Existence of vortex sheets in dimension two], J. Amer. Math. Soc. \textbf{4} (1991), no. 3, 553--586.


\bibitem{dip-lions} R. J. DiPerna and P.-L. Lions, {\it Ordinary 
differential equations, transport theory and Sobolev spaces,}
Invent. Math. \textbf{98} (1989), 511--547.

\bibitem{frenod-sonnendrucker-98} E. Fr\'enod and E. Sonnendr\"ucker, \emph{Homogenization of the Vlasov equation and of the Vlasov-Poisson system with a strong external magnetic field}, Asymptot. Anal. \textbf{18} (1998), no. 3-4, 
193--213.

\bibitem{frenod-sonnendrucker-99}  E. Fr\'enod and E. Sonnendr\"ucker, \emph{Long time behavior of the two-dimensional Vlasov equation with a strong external magnetic field}, Math. Models Methods Appl. Sci. \textbf{10} (2000), no. 4, 539--553.

\bibitem{frenod-sonnendrucker-01}E. Fr\'enod and E. Sonnendr\"ucker, \emph{The Finite Larmor Radius Approximation},
SIAM J. Math. Anal. \textbf{32} (2001), no. 6, 1227--1247.



\bibitem{golse-sr} F. Golse and L. Saint-Raymond, \emph{The Vlasov-Poisson system with strong 
magnetic field}, J. Math. Pures Appl. (9) \textbf{78} (1999), no. 8, 791--817.

\bibitem{golse-sr-2} F. Golse and L. Saint-Raymond, \emph{The Vlasov-Poisson system with strong magnetic
field in quasineutral regime},  Mathematical Models and Methods in Applied
Sciences \textbf{13} (2003), no. 5, 661--714.

\bibitem{grenier}  E. Grenier, \emph{Oscillations in quasineutral plasmas}, Comm. Partial Differential Equations \textbf{21} (1996), no. 3--4, 363--394.

\bibitem{han-kwan} D. Han-Kwan, \emph{The three-dimensional Finite Larmor Radius Approximation}, Asymptot. Anal. \textbf{66} (2010), no.1, 9--33.

\bibitem{hauray-nouri}  M. Hauray and A. Nouri, \emph{Well-posedness of a diffusive gyro-kinetic model},
Ann. Inst. H. Poincar\'e Anal. Non Lin\'eaire \textbf{28} (2011), no. 4, 529--550.



\bibitem{ghn} P. Ghendrih, M. Hauray and A. Nouri, 
\emph{Derivation of a gyrokinetic model. Existence and uniqueness of specific stationary solution}, Kinet. Relat. Models \textbf{2} (2009), no. 4, 707--725.




 \bibitem{LP} P. L. Lions and B. Perthame, \emph{Propagation of moments
and regularity  for the 3-dimensional Vlasov-Poisson system}, Invent.
Math. \textbf{105} (1991), 415--430.


\bibitem{lieb-loss} E. Lieb and M. Loss, \emph{Analysis}, 2nd edition, GSM14, Amer. Math. Soc. Providence RI, 2001.

\bibitem{loeper} G. Loeper,
\emph{Uniqueness  of  the  solution  to  the  Vlasov-Poisson  system  with  bounded
density}, J. Math. Pures Appl. \textbf{86} (9)(2006), no. 1, 68--79.

\bibitem{livrejaune}C. Marchioro and M. Pulvirenti, {\it Mathematical
Theory of Incompressible Nonviscous Fluids}, Springer-Verlag, New York, 1994.

\bibitem{majda-93} A. J. Majda, \emph{Remarks on weak solutions for vortex sheets with a distinguished sign},
Indiana Univ. Math. J.  \textbf{42} (1993), 921-939.

\bibitem{majda-bertozzi} A. J. Majda and A. L. Bertozzi, {\it Vorticity and incompressible flow}, Cambridge Texts in Applied Mathematics \text{27}. Cambridge University Press, Cambridge, 2002.

\bibitem{ukai}  S. Okabe and T. Ukai,
\emph{On classical solutions in the large in time of the two-dimensional Vlasov equation}, Osaka J. Math. \textbf{15} (1978), 245--261.



\bibitem{SR-01}L. Saint-Raymond, \emph{The gyrokinetic approximation for the Vlasov-Poisson system}, Math. Mod. Meth. Appl. Sci. \textbf{10} (9) (2000), 1305--1332.


\bibitem{SR-02} L. Saint-Raymond,  \emph{Control of large velocities in the two-dimensional gyrokinetic approximation},
J. Math. Pures Appl. (9) \textbf{81} (2002), no. 4, 379--399.



\bibitem{Schochet} S. Schochet, 
\emph{The weak vorticity formulation of the 2-D Euler equations and concentration-cancellation},
Comm. Partial Differential Equations \textbf{20} (1995), no. 5--6, 1077--1104.



\end{thebibliography}
\end{document}